\documentclass[10pt]{amsart}
\usepackage{amssymb}
\usepackage{amscd}
\usepackage{comment}

\theoremstyle{plain}
\newtheorem{theorem}{Theorem}
\newtheorem*{mainthm}{Main Theorem}
\theoremstyle{definition}

\newtheorem{proposition}{Proposition}[section]

\newtheorem{corollary}[proposition]{Corollary}
\newtheorem{lemma}[proposition]{Lemma}
\newtheorem{definition}[proposition]{Definition}
\theoremstyle{remark}
\newtheorem{claim}{Claim}

\DeclareMathOperator{\stat}{stat}

\DeclareMathOperator{\nacc}{nacc}
\DeclareMathOperator{\Sk}{Sk}

\DeclareMathOperator{\otp}{otp}
\DeclareMathOperator{\Ch}{Ch}

\DeclareMathOperator{\pr}{Pr}

\DeclareMathOperator{\ran}{ran}

\DeclareMathOperator{\cf}{cf}

\newcommand{\sk}{\vskip.05in}

\newcommand{\restr}{\upharpoonright}

\newcommand{\subs}{\subseteq}

\DeclareMathOperator{\pp}{pp}

\numberwithin{equation}{section}
\begin{document}
\title{A coloring theorem for successors of singular cardinals}
\author{Todd Eisworth}
\address{Department of Mathematics\\
         Ohio University\\
         Athens, OH 45701}
\email{eisworth@math.ohiou.edu}
\date{\today}
\begin{abstract}
We formulate and prove (in {\sf ZFC})  a strong coloring theorem which holds at successors of singular cardinals, and use
it to answer several questions concerning Shelah's principle $\pr_1(\mu^+,\mu^+,\mu^+,\cf(\mu))$ for singular $\mu$.
\end{abstract}

\maketitle
\section{Introduction}

In earlier work (see \cite{upgradesi}), we obtained the following coloring theorem for successors of singular cardinals:
\begin{quote}
Let $\mu$ be singular. There is a function $D:[\mu^+]^2\rightarrow [\mu^+]^2\times\cf(\mu)$ such that for any unbounded $A\subs\mu^+$, there is a stationary $S\subs\mu^+$ with
\begin{equation}
[S]^2\times\cf(\mu)\subs \ran(D\restr[A]^2).
\end{equation}
\end{quote}

This paper arose out of an attempt at finding similar colorings with even stronger properties. This attempt was successful, with the happy consequence that we can use this new coloring theorem to settle a few questions which have arisen in the author's previous work.

Before proceeding any further, we need to alert the reader to one of our notational conventions.

\begin{definition}
If $A$ and $B$ are sets of ordinals, then we define
\begin{equation}
A\circledast B:=\{\langle \alpha,\beta\rangle\in A\times B: \alpha<\beta\}.
\end{equation}
If $C$ is also a set of ordinals, then we define
\begin{equation}
\label{adhoc}
A\circledast B\times C:=\{\langle\alpha,\beta,\gamma\rangle:\langle \alpha,\beta\rangle\in A\circledast B\text{ and }\gamma\in C\}
\end{equation}
\end{definition}

This notation is a bit {\em ad hoc} (especially (\ref{adhoc})), but it makes it much easier to state our main theorem.
Note that if $\kappa$ is a cardinal, then it is quite common to identify $[\kappa]^2$ with those ordered pairs $\langle \alpha,\beta\rangle$ for which $\alpha<\beta$.  In our notation, $[\kappa]^2$ therefore corresponds exactly with $\kappa\circledast\kappa$.

Moving on, we can now state the principal result of this paper:

\begin{mainthm}
Let $\mu$ be a singular cardinal. There is a function
\begin{equation}
D:[\mu^+]^2\rightarrow \mu^+\times\mu^+\times\cf(\mu)
 \end{equation}
 such that whenever $\langle t_\alpha:\alpha<\mu^+\rangle$ is a family of pairwise disjoint members of $[\mu^+]^{<\cf(\mu)}$, there are stationary subsets $S$ and $T$ of $\mu^+$ such that
 whenever
 \begin{equation}
 \langle\alpha,\beta,\delta\rangle \in S\circledast T\times \cf(\mu),
 \end{equation}
  there are $\alpha<\beta<\mu^+$ such that
\begin{equation}
D\restr t_\alpha\times t_\beta\text{ is constant with value }\langle\alpha^*,\beta^*,\delta\rangle.
\end{equation}
\end{mainthm}

At this point, the reader may well be asking ``So what?'', as this theorem seems at first glance to be only a much more technical version of the earlier result, which was quite technical in its own right. This first impression is misleading, though. In the first place, we note that this theorem is a {\sf ZFC} result --- a bit of a rarity given the mystery which still clouds the subject of partition relations at successors of singular cardinals. More importantly, though, it turns that this result is powerful enough to have some important consequences for combinatorial set theory. In particular, it puts us in a position to answer several questions concerning a family of combinatorial principles studied extensively by Shelah in his book~\cite{cardarith}:

\begin{definition}
Suppose $\mu$ is a singular cardinal and $\sigma\leq\mu^+$. The principle $\pr_1(\mu^+,\mu^+,\sigma,\cf(\mu))$
holds if there is a function $f:[\mu^+]^2\rightarrow\sigma$ such that whenever $\langle t_\alpha:\alpha<\lambda\rangle$
is a sequence of pairwise disjoint elements of $[\mu^+]^{<\cf(\mu)}$, then for any $\epsilon<\sigma$ we can find $\alpha<\beta<\mu^+$ such that $f\restr t_\alpha\times t_\beta$ is constant with value $\epsilon$.
\end{definition}

The above definition is clearly only a special case of something much more general (after all, the notation involves five parameters), but we have focused only on the cases of interest to us here.   Note as well that $\pr_1(\mu^+,\mu^+,\theta,\cf(\mu))$ is a very strong version of more standard negative square-brackets partition relations, as the following implications hold:
\begin{equation}
\label{implications}
\pr_1(\mu^+,\mu^+,\theta,\cf(\mu))\Longrightarrow\mu^+\nrightarrow[\mu^+]^2_{\theta}\Longrightarrow \mu^+\nrightarrow [\mu^+]^{<\omega}_\theta.
\end{equation}

What sorts of consequences can we deduce from our main theorem? One example of importance to the author's work is the equivalence of the two statements
\begin{gather}
\label{first1}\pr_1(\mu^+,\mu^+,\mu^+,\cf(\mu))\\
\intertext{and}
\label{second1}\pr_1(\mu^+,\mu^+,\mu,\cf(\mu)).
\end{gather}
The equivalence of these two statements answers a question that has been around since Shelah's original work on~\cite{535} in the mid 1990s. This is important because in many situations it is the first statement which is sought, while the proof only establishes the second.  In some situations, an {\em ad hoc} argument depending on the nature of the function obtained to witness~(\ref{second1}) has been given to allow one to obtain~(\ref{first1}). For example, this is what occurs in~\cite{535}. In other instances, however, the function constructed to verify~(\ref{second1}) did not seem to admit such an upgrade, a set of circumstances which occurred in the author's~\cite{ideals} and Chapter~IV of Shelah's~\cite{cardarith}.

The equivalence of~(\ref{first1}) and~(\ref{second1}) removes these concerns, and allows us to resolve the associated questions.  For example, we can show now that
\begin{equation}
\pp(\mu)=\mu^+\Longrightarrow \pr_1(\mu^+,\mu^+,\mu^+,\cf(\mu)),
\end{equation}
extending a chain of theorems containing results of Erd\H{o}s and Hajnal, Shelah, and Todor{\v{c}}evi{\'c}. Another consequence is that the main theorem of~\cite{ideals} can be fully extended to cover singular cardinals of countable cofinality, and this has consequences for stationary reflection. We will discuss these matters in much more detail in the final section of the paper, as well as obtaining several other related results. Readers willing to accept our main theorem as a ``black box'' can certainly the last section with no problems.

\section{Background material}

The background material we need is almost identical to that required for the results of~\cite{upgradesi}, although we do need to be a little more sophisticated in our use of scales.  The background divides neatly into four separate areas, so we consider each in turn.

\noindent{\sf Minimal Walks}

\smallskip

We must be content with only a brief introduction to Todor{\v{c}}evi{\'c}'s technique of minimal walks and its generalizations. We apology in advance for the notation --- we will be mixing methods of Todor{\v{c}}evi{\'c} together with arguments of Shelah, and so our notation is a hybrid of notations they utilize.

We start by recalling that $\bar{e}=\langle e_\alpha:\alpha<\lambda\rangle$ is a $C$-sequence for the cardinal
$\lambda$ if $e_\alpha$ is closed unbounded in $\alpha$ for each $\alpha<\lambda$.
Given $\alpha<\beta<\lambda$ the {\em minimal walk from $\beta$ to $\alpha$ along $\bar{e}$}
is defined to be the sequence $\beta=\beta_0>\dots>\beta_{n}=\alpha$ obtained by setting
\begin{equation}
\beta_{i+1}=\min(e_{\beta_i}\setminus\alpha).
\end{equation}

We will need to use several functions defined in terms of minimal walks.
For example, we need the function $\rho_2:[\lambda]^2\rightarrow\omega$ giving the length of the walk from $\beta$ to $\alpha$, that is,
\begin{equation}
\rho_2(\alpha,\beta)=\text{ least $i$ for which $\beta_i(\alpha,\beta)=\alpha$}.
\end{equation}
Next, for $i\leq\rho_2(\alpha,\beta)$, we set
\begin{equation*}
\beta_i^-(\alpha,\beta)=
\begin{cases}
0 &\text{if $i=0$},\\
\sup(e_{\beta_j(\alpha,\beta)}\cap\alpha) &\text{if $i=j+1$ for $j<\rho_2(\alpha,\beta)$}.
\end{cases}
\end{equation*}
Note that if $0<i<\rho_2(\alpha,\beta)$ then
\begin{itemize}
\item  $\beta_i^-(\alpha,\beta)=\max(e_{\beta_{i-1}(\alpha,\beta)}\cap\alpha)$ (as opposed to ``sup''),
\sk
\item $\beta_i^-(\alpha,\beta)<\alpha<\beta_i(\alpha,\beta)$, and
\sk
\item $\beta_i(\alpha,\beta)=\min(e_{\beta_{i-1}(\alpha,\beta)}\setminus \beta_i^-(\alpha,\beta)+1)$.
\sk
\end{itemize}
Thus, for $0<i<\rho_2(\alpha,\beta)$, the ordinals $\beta^-_i(\alpha,\beta)$ and $\beta_i(\alpha,\beta)$ are the two consecutive elements in $e_{\beta_{i-1}(\alpha,\beta)}$ which ``bracket''~$\alpha$.

The limiting case where $i=\rho_2(\alpha,\beta)$ turns out to be quite important to us as well.  In this situation,  we have
\begin{itemize}
\sk
\item $\beta_{\rho_2(\alpha,\beta)}^-(\alpha,\beta)\leq\alpha=\beta_{\rho_2(\alpha,\beta)}(\alpha,\beta)$, and
\sk
\item $\beta^-_{\rho_2(\alpha,\beta)}(\alpha,\beta)<\alpha$ if and only if $\alpha\in\nacc(e_{\beta_{\rho_2(\alpha,\beta)-1}(\alpha,\beta)})$.
\end{itemize}
Notice that $\alpha$ must be an element of  $e_{\beta_{\rho_2(\alpha,\beta)-1}(\alpha,\beta)}$ by definition, and  $\beta^-_{\rho_2(\alpha,\beta)}(\alpha,\beta)$ is less than $\alpha$ precisely when $\alpha$ fails to be an accumulation point of $e_{\beta_{\rho_2(\alpha,\beta)-1}(\alpha,\beta)}$ (this should also explain the meaning of ``$\nacc$'').

Continuing our discussion, we define
\begin{gather}
\gamma(\alpha,\beta)=\beta_{\rho_2(\alpha,\beta)-1}(\alpha,\beta),\\
\gamma^-(\alpha,\beta)=\max\{\beta_i^-(\alpha,\beta):i<\rho_2(\alpha,\beta)\},
\intertext{and}
\label{etadef}
\eta(\alpha,\beta)=\max\{\beta_i^-(\alpha,\beta):i\leq\rho_2(\alpha,\beta)\}.
\end{gather}

The following proposition contains some standard facts about minimal walks couched in our notation. The proof is an easy induction.

\begin{proposition}
\label{minimalprop}
Let $\bar{e}$ be a $C$-system, and suppose $\alpha<\beta$.
\begin{enumerate}
\item $\gamma^-(\alpha,\beta)<\alpha$, and if $\gamma^-(\alpha,\beta)<\alpha^*\leq\alpha$ then
\begin{equation}
\beta_i(\alpha,\beta)=\beta_i(\alpha^*,\beta)\text{ for $i<\rho_2(\alpha,\beta)$}.
\end{equation}
\sk
\item $\eta(\alpha,\beta)\leq\alpha$, and if it happens that $\eta(\alpha,\beta)<\alpha^*\leq\alpha$, then
\begin{equation}
\beta_i(\alpha,\beta)=\beta_i(\alpha^*,\beta)\text{ for $i\leq\rho_2(\alpha,\beta)$}.
\end{equation}
In particular,
\begin{equation}
\beta_{\rho_2(\alpha,\beta)}(\alpha^*,\beta)=\alpha.
\end{equation}
\end{enumerate}
\end{proposition}

Note that part (2) of the above proposition is of no interest unless we can guarantee $\eta(\alpha,\beta)<\alpha$ (or equivalently, guarantee $\alpha\in\nacc(e_{\gamma(\alpha,\beta)})$); this will be one of our concerns in the sequel.

Proposition~\ref{minimalprop} captures the only properties of minimal walks we need, and we refer the reader to to~\cite{acta} or~\cite{stevobook} for information on more sophisticated applications.

We do need to use a generalization of the minimal walks machinery in order to handle some issues that arise when dealing with successors of singular cardinals of countable cofinality.  These techniques were introduced by the author and Shelah in~\cite{819}, and they were further developed in~\cite{ideals}.

\begin{definition}
\label{generalizeddef}
Let $\lambda$ be a cardinal. A {\em generalized $C$-sequence} is a family
\begin{equation*}
\bar{e}=\langle e^m_\alpha:\alpha<\lambda, m<\omega\rangle
\end{equation*}
such that for each $\alpha<\lambda$ and $m<\omega$,
\begin{itemize}
\item $e^m_\alpha$ is closed unbounded in $\alpha$, and
\sk
\item $e^m_\alpha\subs e^{m+1}_\alpha$.
\sk
\end{itemize}
\end{definition}

One can think of a generalized $C$-sequence as a countable family of $C$-sequences which are increasing in a sense. One can also utilize generalized $C$-sequences in the context of minimal walks. In this paper, we do this in the simplest fashion --- given $m<\omega$ and $\alpha<\beta<\lambda$, we let the {\em $m$-walk from $\beta$ to $\alpha$ along $\bar{e}$} consist of the minimal walk from $\beta$ to $\alpha$ using the $C$-sequence $\langle e^m_\gamma:\gamma<\lambda\rangle$. Such walks have their associated parameters, and we use the superscript $m$ to indicate which part of the generalized $C$-sequence is being used in computations. So, for example, the $m$-walk from $\beta$ to $\alpha$ along $\bar{e}$ will have length $\rho_2^m(\alpha,\beta)$, and consist of ordinals denoted $\beta^m_i(\alpha,\beta)$ for $i\leq\rho^m_2(\alpha,\beta)$.

The requirement that $e^m_\alpha\subs e^{m+1}_\alpha$ is relevant for the following reason.  Given $\alpha<\beta$, we note that the sequence $\langle \min(e_\beta^m\setminus\alpha):m<\omega\rangle$ is non-increasing, and therefore eventually constant. From this it follows easily that the $m$-walk from $\beta$ to~$\alpha$  along $\bar{e}$ is exactly the same for all sufficiently large $m$, that is, there is an $m^*$ such that
\begin{equation}
\label{largeenoughm}
\langle \beta^m_i(\alpha,\beta):i<\rho^m_2(\alpha,\beta)\rangle = \langle \beta^{m^*}_i(\alpha,\beta):i<\rho^{m^*}_2(\alpha,\beta)\rangle
\end{equation}
whenever $m\geq m^*$.

\medskip

\noindent{\sf Club-guessing}

\medskip

Just as in~\cite{upgradesi}, we are going to need generalized $C$-sequences that have been carefully selected to interact with certain club-guessing sequences.  The sort of  club-guessing sequence we use depends on whether or not the cofinality of our singular cardinal~$\mu$ is uncountable, and we will handle each case separately. In either case, we will be defining a stationary set $S\subs \mu^+$, a club-guessing sequence $\bar{C}$, and a generalized $C$-sequence $\bar{e}$.

If the cofinality of $\mu$ is uncountable, then we define
\begin{equation}
S:= S^{\mu^+}_{\cf(\mu)}=\{\delta<\mu^+:\cf(\delta)=\cf(\mu)\}.
\end{equation}
Claim~2.6 on page~127 of~\cite{cardarith} (or see Theorem~2 of \cite{819}) gives us a sequence $\langle C_\delta:\delta\in S\rangle$ such that
\begin{itemize}
\sk
\item $C_\delta$ is club in $\delta$,
\sk
\item $\otp(C_\delta)=\cf(\mu)$,
\sk
\item $\langle\cf(\alpha):\alpha\in\nacc(C_\delta)\rangle$ increases to $\mu$, and
\sk
\item whenever $E$ is club in~$\mu^+$, there are stationarily many $\delta\in S$ for which $C_\delta\subs E$.
\sk
\end{itemize}
(Recall ``$\nacc(C_\delta)$'' refers to the non-accumulation points of $C_\delta$, that is, those elements of $C_\delta$ that are not limits of points in $C_\delta$.)

We now use the ``ladder swallowing'' trick (see Lemma~13 of~\cite{nsbpr}) to build a $C$-sequence $\langle e_\alpha:\alpha<\mu^+\rangle$ such that for each $\alpha<\mu^+$,
\begin{gather}
|e_\alpha|< \mu,\\
\intertext{and}
\delta\in S\cap e_\alpha\Longrightarrow C_\delta\subs e_\alpha.
\end{gather}
We then construct a (admittedly somewhat trivial) generalized $C$-sequence $\bar{e}=\langle e^m_\alpha:m<\omega,\alpha<\mu^+\rangle$ by setting $e^m_\alpha = e_\alpha$ for all $m<\omega$.

In the case where $\mu$ is of countable cofinality, our definition of $S$, $\bar{C}$, and $\bar{e}$ is a little more involved because of some open questions concerning club-guessing. A reader interested in these issues can find a more detailed discussion in~\cite{819}, but we shall rely on technology developed in~\cite{ideals}.

Start by setting
\begin{equation}
S:= S^{\mu^+}_{\aleph_1}=\{\delta<\mu^+:\cf(\delta)=\aleph_1\},
\end{equation}
and assume $\langle \mu_m:m<\omega\rangle$ is an increasing sequence of uncountable cardinals cofinal in~$\mu$.

We are going to present a simplified version of the conclusion of Theorem~4 of~\cite{ideals}; the reader can consult that paper for a detailed proof (Proposition~5.8 is particularly relevant).
In particular, the work in~\cite{ideals} provides us with a sequence $\langle C_\delta:\delta\in S\rangle$ such that each $C_\delta$ is club in~$\delta$, and $C_\delta=\bigcup_{m<\omega}C_\delta[m]$ where
\begin{gather}
C_\delta[m]\text{ is closed and unbounded in }\delta,\\
|C_\delta[m]|\leq\mu^+_m,
\end{gather}
and such that for every club $E\subs\mu^+$, there are stationarily many $\delta\in S$ such that for each $m<\omega$, $\nacc(C_\delta[m])\cap E$ contains unboundedly many ordinals of cofinality greater than $\mu_m^+$.  (Note that the use of ``$\nacc$'' is redundant as the cofinality assumption guarantees such an ordinal cannot be a limit point of $C_\delta[m]$.)

An application of Lemma~5.10 from~\cite{ideals} provides us with a generalized $C$-sequence $\bar{e}=\langle e_\alpha^m:\alpha<\mu^+,m<\omega\rangle$ satisfying
\begin{gather}
\label{2.16}|e^m_\alpha|\leq\cf(\alpha)+\mu_m^+\\
\intertext{and}
\label{2.17}\delta\in S\cap e_\alpha^m\Longrightarrow C_\delta[m]\subs e_\alpha^m.
\end{gather}

No matter what the cofinality of $\mu$, the reader should give  the phrase ``choose $\delta\in S$ such that $C_\delta$ guesses $E$'' the obvious interpretation in light of the above discussion.

\medskip

\noindent{\sf Scales}

\medskip

The next ingredient we need for our theorem is the concept of a scale for a singular cardinal.

\begin{definition}
Let $\mu$ be a singular cardinal. A {\em scale for
$\mu$} is a pair $(\vec{\mu},\vec{f})$ satisfying
\begin{enumerate}
\item $\vec{\mu}=\langle\mu_i:i<\cf(\mu)\rangle$ is an increasing sequence of regular cardinals
such that $\sup_{i<\cf(\mu)}\mu_i=\mu$ and $\cf(\mu)<\mu_0$.
\item $\vec{f}=\langle f_\alpha:\alpha<\mu^+\rangle$ is a sequence of functions such that
\begin{enumerate}
\item $f_\alpha\in\prod_{i<\cf(\mu)}\mu_i$.
\item If $\gamma<\delta<\beta$ then $f_\gamma<^* f_\beta$, where  the notation $f<^* g$  means that $\{i<\cf(\mu): g(i)\leq f(i)\}$ is bounded in $\cf(\mu)$.
\item If $f\in\prod_{i<\cf(\mu)}\mu_i$ then there is an $\alpha<\beta$ such that $f<^* f_\alpha$.
\end{enumerate}
\end{enumerate}
\end{definition}

Of course, we will be using the important theorem of Shelah (see  Main Claim~1.3 on page~46 of~\cite{cardarith}) that scales exist for any singular $\mu$.  Readers seeking a gentler exposition of this and related topics can consult~\cite{cummings}, or~\cite{myhandbook}.

We are going to need a result from~\cite{nsbpr} concerning scales.  We remind the reader that notation of the form ``$(\exists^*\beta<\lambda)\psi(\beta)$''
  means $\{\beta<\lambda:\psi(\beta)\text{ holds}\}$ is unbounded below~$\lambda$, while
    ``$(\forall^*\beta<\lambda)\psi(\beta)$'' means that $\{\beta<\lambda:\psi(\beta)\text{ fails}\}$ is bounded
     below~$\lambda$.

\begin{lemma}
\label{scalelemma}
Let $\mu$ be singular,  and suppose $(\vec{\mu},\vec{f})$ is a scale
for~$\mu$.  Then there is a closed unbounded $C\subs\mu^+$ such that the following holds for every $\beta\in C$:
\begin{equation}
\label{2.18}
(\forall^*i<\cf(\mu))(\forall\eta<\mu_i)(\forall\nu<\mu_{i+1})(\exists^*\alpha<\beta)\left[f_\alpha(i)>\eta\wedge f_\alpha(i+1)>\nu\right].
\end{equation}
\end{lemma}
\begin{proof}
See Lemma 7 of~\cite{nsbpr}.
\end{proof}
\medskip

\noindent{\sf Elementary Submodels}

\medskip

We have the usual conventions when dealing with elementary submodels. For example, we always assume that $\chi$ is regular cardinal much larger than anything relevant to discussion at hand, and we  let $\mathfrak{A}$ denote the structure $\langle H(\chi),\in, <_\chi\rangle$ where $H(\chi)$ is the collection of sets hereditarily of cardinality less than $\chi$, and $<_\chi$ is some suitable well-order of $H(\chi)$. We include $<_\chi$ in our structure so that we can talk about Skolem hulls, as the well-ordering gives us definable Skolem functions. In general, if $B\subs H(\chi)$, then we denote the Skolem hull of $B$ in $\mathfrak{A}$ by $\Sk_{\mathfrak{A}}(B)$.

The following result of Baumgartner~\cite{jb} is critical:

\begin{lemma}
\label{newcharlem}
Assume that $M\prec\mathfrak{A}$ and let $\sigma\in M$ be a cardinal.  If we define $N=\Sk_{\mathfrak{A}}(M\cup\sigma)$
then for all regular cardinals $\tau\in M$ greater than $\sigma$, we have
\begin{equation*}
\sup(M\cap\tau)=\sup(N\cap\tau).
\end{equation*}
\end{lemma}
\begin{proof}
See the last section of~\cite{myhandbook}, or Lemma~9 of~\cite{nsbpr}.
\end{proof}

The above lemma gives us a crucial fact about {\em characteristic functions} of models,
which we define next.

\begin{definition}
\label{chardef}
Let $\mu$ be a singular cardinal, and let $\vec{\mu}=\langle\mu_i:i<\cf(\mu)\rangle$ be an increasing
sequence of regular cardinals cofinal in $\mu$.  If $M$ is an elementary submodel of $\mathfrak{A}$
such that
\begin{itemize}
\item $|M|<\mu$,
\item $\langle \mu_i:i<\cf(\mu)\rangle\in M$, and
\item $\cf(\mu)+1\subs M$,
\end{itemize}
then the {\em characteristic function of $M$ on $\vec{\mu}$} (denoted $\Ch^{\vec{\mu}}_M$) is the function
with domain $\cf(\mu)$ defined by
\begin{equation*}
\Ch^{\vec{\mu}}_M(i):=
\begin{cases}
\sup(M\cap\mu_i) &\text{if $\sup(M\cap\mu_i)<\mu_i$,}\\
0  &\text{otherwise.}
\end{cases}
\end{equation*}
If $\vec{\mu}$ is clear from context, then we suppress reference to it in the notation.
\end{definition}

In the above situation, it is clear that $\Ch^{\vec{\mu}}_M$ is an element of the product
$\prod_{i<\kappa}\mu_i$, and furthermore, $\Ch^{\vec{\mu}}_M(i)=\sup(M\cap\mu_i)$ for all
sufficiently large $i<\cf(\mu)$. We obtain the following as an immediate consequence of Lemma~\ref{newcharlem}.

\begin{corollary}
\label{skolemhulllemma}
Let $\mu$, $\vec{\mu}$, and $M$ be as in Definition~\ref{chardef}.
If $i^*<\cf(\mu)$ and we define~$N$ to be $\Sk_{\mathfrak{A}}(M\cup\mu_{i^*})$,
then
\begin{equation}
\Ch^{\vec{\mu}}_M\restr [i^*+1,\cf(\mu))=\Ch^{\vec{\mu}}_N\restr [i^*+1,\cf(\mu)).
\end{equation}
\end{corollary}

We need one final bit of notation concerning elementary submodels.

\begin{definition}
Let $\lambda$ be a regular cardinal. A $\lambda$-approximating sequence is a continuous $\in$-chain
$\mathfrak{M}=\langle M_i:i<\lambda\rangle$ of elementary submodels of $\mathfrak{A}$ such that
\begin{enumerate}
\item $\lambda\in M_0$,
\item $|M_i|<\lambda$,
\item $\langle M_j:j\leq i\rangle\in M_{i+1}$, and
\item $M_i\cap\lambda$ is a proper initial segment of $\lambda$.
\end{enumerate}
If $x\in H(\chi)$, then we say that $\mathfrak{M}$ is a $\lambda$-approximating sequence over $x$ if
$x\in M_0$.
\end{definition}

Note that if $\mathfrak{M}$ is a $\lambda$-approximating sequence and $\lambda=\mu^+$, then $\mu+1\subs M_0$ because
of condition~(4) and the fact that $\mu$ is an element of each $M_i$.

\section{The minimal walk lemma}

This section begins with a technical {\em ad hoc} definition. The definition actually depends on a generalized $C$-system $\bar{e}$, a scale $(\vec{\mu},\vec{f})$, and a sequence $\langle t_\alpha:\alpha<\mu^+\rangle$ of pairwise disjoint elements of $[\mu^+]^{<\cf(\mu)}$, but we suppress this dependence in the notation.

\begin{definition}
Suppose $k<\cf(\mu)$ and $m<\omega$.
The formula $\psi_{k,m}(\eta,\eta^+,\beta^*,\gamma,\beta)$ asserts
\begin{gather}
\eta\leq\eta^+<\beta^*<\gamma\leq\min(t_\beta),\\
\eta=\min\{\eta^m(\beta^*,\epsilon):\epsilon\in t_\beta\},\\
\eta^+=\sup\{\eta^m(\beta^*,\epsilon):\epsilon\in t_\beta\},\\
\gamma=\min\{\gamma^m(\beta^*,\epsilon):\epsilon\in t_\beta\},\\
\intertext{and}
(\forall\epsilon\in t_\beta)\left[f_\eta\restr [k,\cf(\mu))\leq f_{\eta^m(\beta^*,\epsilon)}\restr [k,\cf(\mu))\leq f_{\eta^+}\restr [k,\cf(\mu))\right].
\end{gather}
\end{definition}

The above definition probably seems quite mysterious, but it isolates a certain configuration whose importance will become clear as the proof progresses. The next result is one of two technical lemmas crucial for our arguments.

\begin{lemma}
\label{mainlemma}
Suppose $\mu$ is a singular cardinal, and $S$, $\bar{C}$, and $\bar{e}$ are as in the preceding section. Further suppose that $(\vec{\mu},\vec{f})$ is a scale for $\mu$ and $\langle t_\beta:\beta<\mu^+\rangle$ is a sequence of pairwise disjoint elements of $[\mu^+]^{<\cf(\mu)}$. Then there are an $m<\omega$ and $k<\cf(\mu)$ such that
\begin{multline}
(\exists^*\eta<\mu^+)(\exists\eta^+<\mu^+)(\exists^{\stat}\beta^*<\mu^+)\\
(\exists^*\gamma<\mu^+)(\exists\beta<\mu^+)[\psi_{k,m}(\eta,\eta^+,\beta^*,\gamma,\beta)].
\end{multline}
\end{lemma}

\begin{proof}
We start by defining $x:=\{\mu, S,(\vec{\mu},\vec{f})),\bar{C},\bar{e},\langle t_\beta:\beta<\mu^+\rangle\}$, that is, $x$ consists of the parameters needed to comprehend $\psi$. We then let $\langle M_\alpha:\alpha<\mu^+\rangle$ be a $\mu^+$-approximating sequence over $x$.  It is clear that
the set
\begin{equation}
E:=\{\delta<\mu^+:M_\delta\cap\mu^+=\delta\}
\end{equation}
is closed and unbounded in $\mu^+$, so can fix $\delta\in S$ for which $C_\delta$ guesses $E$.

Choose $\beta$ such that
\begin{equation}
\delta<\min(t_\beta),
\end{equation}
and define
\begin{equation}
\gamma^{\otimes}=\sup\{\gamma^{m,-}(\delta,\epsilon):\epsilon\in t_\beta, m<\omega\}.
\end{equation}

We observe
\begin{equation}
\max\{\gamma^\otimes,M_0\cap\mu^+\}<\delta,
\end{equation}
as $\delta\in E$ and $|t_\beta|+\aleph_0<\cf(\delta)$.

\begin{proposition}
\label{claim1}
We can find $\beta^*<\delta$ and $m<\omega$ such that
\begin{gather}
\label{3.11}\beta^*\in\nacc(C_\delta[m])\cap E,\\
\label{3.12}\cf(\beta^*)>\cf(\mu)+\sup\{|e^m_{\gamma^m(\delta,\epsilon)}|:\epsilon\in t_\beta\},\\
\intertext{and}
\label{3.13}\max\{\gamma^\otimes, M_0\cap\mu^+\}<\max(C_\delta[m]\cap\beta^*).
\end{gather}
\end{proposition}
\begin{proof}
Our first observation is that for any $m<\omega$, we have
\begin{equation}
\cf(\mu)+\sup\{|e^m_{\gamma^m(\delta,\epsilon)}|:\epsilon\in t_\beta\}<\mu
\end{equation}
as $|t_\beta|<\cf(\mu)$.

If $\mu$ is of uncountable cofinality, then we set $m=1$ (recall that $\bar{e}$ is somewhat trivial in this case). Our choice of $\bar{C}$ and $\delta$ tells us that all sufficiently large elements of $\nacc(C_\delta)$ satisfy~(\ref{3.11}) and~(\ref{3.12}), and so we can easily find $\beta^*$ satisfying~(\ref{3.13}) as well.

If the cofinality of $\mu$ is countable, then we note that $t_\beta$ is finite. Thus, we can fix a single $m_0<\omega$ such that
\begin{equation}
\label{3.15}
\langle\beta^m_i(\delta,\epsilon):i<\rho^m_2(\delta,\epsilon)\rangle = \langle \beta^{m_0}_i(\delta,\epsilon):i<\rho^{m_0}_2(\delta,\epsilon)\rangle
\end{equation}
whenever $\epsilon\in t_\beta$ and $m\geq m_0$.

Now choose $m\geq m_0$ such that
\begin{equation}
\mu^+_m\geq\max\{\cf(\gamma^{m_0}(\delta,\epsilon)):\epsilon\in t_\beta\}.
\end{equation}
In light of (\ref{3.15}), we see
\begin{equation}
\mu^+_m\geq\max\{\cf(\gamma^m(\delta,\epsilon)):\epsilon\in t_\beta\}.
\end{equation}

Now $\nacc(C_\delta[m])\cap E$ contains unboundedly many elements of cofinality greater than $\mu_m^+$, so we can find a $\beta^*$ with the required properties.
\end{proof}

\begin{proposition}
\label{claim2}
For any $\epsilon\in t_\beta$, we have
\begin{gather}
\label{3.18}(\forall i<\rho^m_2(\delta,\epsilon)\left[\beta^m_i(\beta^*,\epsilon)=\beta^m_i(\delta,\epsilon)\right],\\
\label{3.19}\rho^2_m(\beta^*,\epsilon)=\rho^2_m(\delta,\epsilon),\\
\label{3.20}\gamma^m(\beta^*,\epsilon)=\gamma^m(\gamma,\epsilon),\\
\intertext{and}
\label{3.21}\beta^*\in\nacc(e^m_{\gamma^m(\delta,\epsilon)}).
\end{gather}
\end{proposition}
\begin{proof}
The first statement holds as $\gamma^{\otimes}<\beta^*<\delta$. Given $\epsilon\in t_\beta$, we note that $\delta\in e^m_{\gamma^m(\delta,\epsilon)}$ by definition, and so $C_\delta[m]\subs e^m_{\gamma^m(\delta,\epsilon)}$ by our choice of $\bar{e}$. It follows that $\beta^*\in e^m_{\gamma^m(\delta,\epsilon)}$ and both~(\ref{3.19}) and~(\ref{3.20}) are immediate. Given that $\beta^*\in e^m_{\gamma^m(\delta,\epsilon)}$ was established earlier in the proof, the statement (\ref{3.21}) follows from~(\ref{3.12}).
\end{proof}

Now that we have isolated~$\beta^*$ and $m$, we define
\begin{align}
\eta:=&\min\{\eta^m(\beta^*,\epsilon):\epsilon\in t_\beta\},\\
\eta^+:=&\sup\{\eta^m(\beta^*,\epsilon):\epsilon\in t_\beta\},\\
\intertext{and}
\gamma:=&\min\{\gamma^m(\beta^*,\epsilon):\epsilon\in t_\beta\}.
\end{align}

\begin{proposition}
We have
\begin{equation}
M_0\cap\mu^+<\eta\leq\eta^+<\beta^*=M_{\beta^*}\cap\mu^+<\delta=M_\delta\cap\mu^+<\gamma.
\end{equation}
\end{proposition}
\begin{proof}
It is clear that $\eta\leq\eta^+\leq\beta^*<\delta$.
By our choice of $\beta^*$ we know
\begin{equation*}
M_0\cap\mu^+<\max(C_\delta[m]\cap\beta^*).
\end{equation*}
Since $C_\delta[m]\subs e^m_{\gamma^m(\delta,\epsilon)}=e^m_{\gamma^m(\beta^*,\epsilon)}$ for each $\epsilon\in t_\beta$, it follows that
\begin{equation*}
M_0\cap\mu^+<\max(C_\delta[m]\cap\beta^*)\leq\eta.
\end{equation*}
 By (\ref{3.20}) and~(\ref{3.21}), we know that $\eta^m(\beta^*,\epsilon)<\beta^*$ for every $\epsilon\in t_\beta$. Since $\cf(\beta^*)>|t_\beta|$, it follows that $\eta^+<\beta^*$.

Finally, (\ref{3.20}) and the definition of $\gamma$ imply that $\delta<\gamma$, and the proof is complete.
\end{proof}

Now let $A:=\{\eta^m(\beta^*,\epsilon):\epsilon\in t_\beta\}$. It is clear that $|A|<\cf(\mu)$ so there must exist a single $k<\cf(\mu)$ such that
\begin{equation*}
\xi_0<\xi_1\text{ in }A\Longrightarrow f_{\xi_0}\restr [k, \cf(\mu))< f_{\xi_1}\restr[k, \cf(\mu)).
\end{equation*}

Given this choice of $k$ together with the preceding work, we see that
\begin{equation}
\psi_{k,m}(\eta,\eta^+,\beta^*,\gamma,\beta)
\end{equation}
holds.

We now finish the proof of the lemma by a standard elementary submodel argument.  Since $x\cup\{\eta,\eta^+,\beta^*,k,m\}\in M_\delta$ and $\gamma>M_\delta\cap\mu^+$, we know
\begin{equation*}
(\exists^*\gamma<\mu^+)(\exists\beta<\mu^+)[\psi_{k, m}(\eta,\eta^+,\beta^*,\gamma,\beta)].
\end{equation*}
Since $x\cup\{\eta,\eta^+, k, m\}\in M_{\beta^*}$ and $\beta^*=M_{\beta^*}\cap\mu^+$, we know
\begin{equation*}
(\exists^{\stat}\beta^*<\mu^+)(\exists^*\gamma<\mu^+)(\exists^\beta<\mu^+)[\psi_{k, m}(\eta,\eta^+,\beta^*,\gamma,\beta)].
\end{equation*}
Finally, since $M_0\cap\mu^+<\eta$ and $x\cup\{k, m\}\in M_0$, we conclude
\begin{equation*}
(\exists^*\eta<\mu^+)(\exists\eta^+<\mu^+)(\exists^{\stat}\beta^*<\mu^+)(\exists^*\gamma<\mu^+)(\exists^\beta<\mu^+)[\psi_{k, m}(\eta,\eta^+,\beta^*,\gamma,\beta)],
\end{equation*}
as required.

\end{proof}

\section{The $\Gamma$ lemma}

This section makes heavy use of scales and the associated combinatorics, so let us begin by assuming $(\vec{\mu},\vec{f})$ is a scale
for some singular cardinal~$\mu$. There is a natural (and well-known) function $\Gamma:[\mu^+]^2\rightarrow\cf(\mu)$ associated
with our scale, namely
\begin{equation}
\Gamma(\alpha,\beta):=\sup\{i<\cf(\mu):f_\beta(i)\leq f_\alpha(i)\}.
\end{equation}
Notice that $\Gamma(\alpha,\beta)$ is defined whenever $\alpha<\beta<\mu^+$ because $(\vec{\mu},\vec{f})$ is a scale. We
make the convention that $\Gamma(\alpha,\alpha)$ is defined and equal to some symbol $\infty$ as well.

We also define a partial function $\Gamma^+:[\mu^+]^2\rightarrow\cf(\mu)$ by setting
\begin{equation}
\Gamma^+(\alpha,\beta):= \max\{i<\cf(\mu):f_\beta(i)\leq f_\alpha(i)\},
\end{equation}
should this maximum exist, and leaving $\Gamma^+$ undefined in all other situations.  Clearly $\Gamma$ and $\Gamma^+$ agree whenever
the latter function is defined.

\begin{lemma}
\label{Gammalemma}
Assume $\mu$ is a singular cardinal and $(\vec{\mu},\vec{f})$ is a scale for $\mu$.
Further assume:
\begin{itemize}
\item $M_0\in M_1\in M_2$ are elementary submodels of $\mathfrak{A}$ of cardinality~$\mu$ such that $M_i\cap\mu^+$ is an initial segment of~$\mu^+$,
\sk
\item $(\vec{\mu},\vec{f})\in M_0$,
\sk
\item $\beta^*=M_0\cap\mu^+$,
\sk
\item $\bar{s}=\langle s_\alpha:\alpha<\mu^+\rangle$ is sequence of pairwise disjoint elements of $[\mu^+]^{<\cf(\mu)}$ with $\bar{s}\in M_0$, and
\sk
\item $t\in [\mu^+]^{<\cf(\mu)}$ with $M_2\cap\mu^+\leq\min(t)$.
\sk
\end{itemize}
Then for all sufficiently large $i<\cf(\mu)$, there are unboundedly many $\alpha<\beta^*$ such that for all $\epsilon_a\in s_\alpha$ and $\epsilon_b\in t$, we have
\begin{equation}
\Gamma^+(\epsilon_a,\epsilon_b)=i,
\end{equation}
but
\begin{equation}
f_{\beta^*}(i+1)<f_{\epsilon_a}(i+1).
\end{equation}
\end{lemma}
\begin{proof}
Our first observation is that it suffices to prove the lemma under the following assumptions about $\bar{s}$:
\begin{itemize}
\sk
\item $\alpha<\beta\Longrightarrow \sup(t_\alpha)<\min(t_\beta)$,
\sk
\item $\alpha\leq\min(t_\alpha)$, and
\sk
\item there is an $i_0<\cf(\mu)$ such that for all $\alpha<\mu^+$,
\begin{equation}
f_{\min(t_\alpha)}\restr [i_0,\cf(\mu))\leq f_\epsilon\restr [i_0,\cf(\mu))\text{ for all }\epsilon\in t_\alpha.
\end{equation}
\end{itemize}
The point is that given a sequence $\bar{s}$ as in the assumptions of the lemma, we can find an unbounded $J\subs \mu^+$ such that $\langle s_\alpha:\alpha\in J\rangle$ enjoys the three additional properties we want. (For the last property, we need to use that $|s_\alpha|<\cf(\mu)$ for each $\alpha$.)  Thus, there will be such a $J$ in $M_0$, and obtaining the conclusion for $\langle s_\alpha:\alpha\in J\rangle\in M_0$ is enough.

Next, we observe that if we define $\vec{g}=\langle g_\alpha:\alpha<\mu^+\rangle$
where
\begin{equation}
g_\alpha:= f_{\min(t_\alpha)},
\end{equation}
then $(\vec{\mu},\vec{g})$ forms a scale for~$\mu$. This new scale is definable from parameters in~$M_0$, and hence $(\vec{\mu},\vec{g})$ is an element of $M_0$ as well.
In particular, we can apply Lemma~\ref{scalelemma} to $(\vec{\mu},\vec{g})$ in the model $M_0$. Taken in conjunction with the fact that $\beta^*=M_0\cap\mu^+$, we see that there is an $i_1<\cf(\mu)$ such that whenever $i_1\leq i<\cf(\mu)$, $\xi_0<\mu_i$, and $\xi_1<\mu_{i+1}$,
 \begin{equation}
 \label{sentence}
 (\exists^*\alpha<\beta^*)[f_{\min(t_\alpha)}(i)>\xi_0\wedge f_{\min(t_\alpha)}(i+1)>\xi_1].
\end{equation}

The main part of our argument relies on an application of Lemma~\ref{skolemhulllemma}.  Let $C$ be a closed unbounded subset of $\beta^*$ of order-type $\cf(\beta^*)$, and define
\begin{equation}
x=\{\mu,(\vec{\mu},\vec{f}),\bar{s},\beta^*\}\cup\cf(\mu)\cup C.
\end{equation}
We define $M$ to be the Skolem hull of $x$ in the structure $\mathfrak{A}$.

Note that $M$ is of cardinality $\max\{\cf(\beta^*),\cf(\mu)\}$, and $M\in M_2$ as $M$ can be computed in $M_2$ by taking the hull of $x$ using the restrictions of the Skolem functions to $M_1$.  Let $g$ denote the characteristic function of $M$ in $\vec{\mu}$, that is,
\begin{equation}
g=\Ch^{\vec{\mu}}_M.
\end{equation}
It is clear that
\begin{equation}
g\in M_2\cap\prod_{i<\cf(\mu)}\mu_i.
\end{equation}
Let $i_2<\cf(\mu)$ be least with $|M|<\mu_{i_2}$, so whenever $i_2\leq i<\cf(\mu)$ we have
\begin{equation}
g(i)=\sup(M\cap\mu_i).
\end{equation}
Note as well that $C\subs M$ and so $M\cap\beta^*$ is unbounded in $\beta^*$.

Since $g\in M_2\cap\prod_{i<\cf(\mu)}\mu_i$, it follows that $g<^*f_\gamma$ whenever $M_2\cap\mu^+\leq\gamma<\mu^+$.  Since $|t|<\cf(\mu)$, we can fix  $i_3<\cf(\mu)$ such that
\begin{equation}
g\restr [i_3,\cf(\mu))<f_{\epsilon}\restr [i_3,\cf(\mu))\text{ for all }\epsilon\in t.
\end{equation}

Now define
\begin{equation}
i^*=\max\{i_0,i_1,i_2,i_3\}.
\end{equation}
We claim that whenever $i^*\leq i<\cf(\mu)$, there are unboundedly many $\alpha<\beta^*$ for which the conclusion of the lemma holds.

Let such an~$i$ be given, and define
\begin{equation}
N:=\Sk_{\mathfrak{A}}(M\cup\mu_i),
\end{equation}
and
\begin{equation}
h:=\Ch^{\vec{\mu}}_{N}.
\end{equation}
Recall that by Corollary~\ref{skolemhulllemma}, we know
\begin{equation}
\label{samerestriction}
g\restr [i+1,\cf(\mu))= h\restr[i+1,\cf(\mu)).
\end{equation}

Now let us define
\begin{equation}
\xi_0:=\sup\{f_\epsilon(i):\epsilon\in t\},
\end{equation}
and
\begin{equation}
\xi_1:= f_{\beta^*}(i+1).
\end{equation}
Both $\xi_0$ and $\xi_1$ are in $N$ --- $\xi_0$ gets in as $\mu_i\subs N$, while $\xi_1$ is in $N$ because $\beta^*$ and the other parameters needed to define it are in $N$.

The next two claims constitute the heart of the matter:

\begin{claim}
Suppose $\alpha\in N\cap\beta^*$ satisfies
\begin{itemize}
\item $f_{\min(s_\alpha)}(i)>\xi_0$, and
\sk
\item $f_{\min(s_\alpha)}(i+1)>\xi_1$.
\sk
\end{itemize}
Then for all $\epsilon_a\in s_\alpha$, we have
\begin{itemize}
\sk
\item $\Gamma^+(\epsilon_a,\epsilon_b)=i$ for all $\epsilon_b\in t$, and
\sk
\item $f_{\beta^*}(i+1)<f_{\epsilon_a}(i+1)$.
\sk
\end{itemize}
\end{claim}
\begin{proof}
Let $\alpha\in N\cap\beta^*$ satisfy the assumptions of the claim, and fix $\epsilon_a\in s_\alpha$.
Since $i_0\leq i$, we know
\begin{equation}
f_{\beta^*}(i+1)=\xi_1<f_{\min(s_\alpha)}(i+1)\leq f_{\epsilon_a}(i+1),
\end{equation}
and therefore $\Gamma(\epsilon_a,\beta^*)\neq i$.

Now fix $\epsilon_b\in t$. Since $i_0\leq i$, we know
\begin{equation}
f_{\epsilon_b}(i)\leq \xi_0 < f_{\min(s_\alpha)}(i)\leq f_{\epsilon_a}(i),
\end{equation}
and therefore
\begin{equation}
\label{first}
\Gamma^+(\epsilon_a,\epsilon_b)= i.
\end{equation}

Since $s_\alpha\in N$ and $|s_\alpha|<\cf(|\mu|)\subs N$, we know that $\epsilon_a\in N$.  Since $i\geq i_2$, it follows that
\begin{equation}
f_{\epsilon_a}\restr [i+1,\cf(\mu))\leq h\restr [i+1,\cf(\mu)).
\end{equation}
Since $i\geq i_3$, a glance at~(\ref{samerestriction}) shows us that
\begin{equation}
f_{\epsilon_a}\restr [i+1,\cf(\mu))\leq g\restr [i+1,\cf(\mu))<f_{\epsilon_b}\restr [i+1,\cf(\mu)),
\end{equation}
and therefore
\begin{equation}
\label{second}
\Gamma(\epsilon_a,\epsilon_b)\leq i.
\end{equation}
The conjunction of~(\ref{first}) and~(\ref{second}) finishes the proof.
\end{proof}
\begin{claim}
The set of $\alpha\in N\cap\beta^*$ satisfying the hypotheses of the preceding claim is unbounded in $\beta^*$.
\end{claim}
\begin{proof}
The scale $(\vec{\mu},\vec{g})$ is in $N$, and so~(\ref{sentence}) holds in $N$ by elementarity.  We have explicitly ensured that $N\cap\beta^*$ is unbounded in $\beta^*$ (as we demanded $C\subs M\subs N$), and the result follows immediately after another application of elementarity.
\end{proof}
\end{proof}

\section{Main Theorem}
We come now to the proof of the main theorem stated in the introduction.

\begin{theorem}[Main Theorem]
\label{mainthm}
Let $\mu$ be a singular cardinal. There is a function
\begin{equation}
D:[\mu^+]^2\rightarrow \mu^+\times\mu^+\times\cf(\mu)
 \end{equation}
 such that whenever $\langle t_\alpha:\alpha<\mu^+\rangle$ is a family of pairwise disjoint members of $[\mu^+]^{<\cf(\mu)}$, there are stationary subsets $S$ and $T$ of $\mu^+$ such that
 whenever
 \begin{equation}
 \langle\alpha,\beta,\delta\rangle \in S\circledast T\times \cf(\mu),
 \end{equation}
  there are $\alpha<\beta<\mu^+$ such that
\begin{equation}
D\restr t_\alpha\times t_\beta\text{ is constant with value }\langle\alpha^*,\beta^*,\delta\rangle.
\end{equation}
\end{theorem}
\begin{proof}
Let $(\vec\mu,\vec f)$ be a scale for~$\mu$, and fix a generalized $C$-sequence
 $\bar{e}$ as in Lemma~\ref{mainlemma}.
Let $\iota:\cf(\mu)\rightarrow \omega\times\cf(\mu)$ be a function such that for any $m<\omega$
 and $\delta<\cf(\mu)$, there are unboundedly many $i<\cf(\mu)$ with $\iota(i)=\langle m,\delta\rangle$.

Given $\alpha<\beta$, we let $m$ and $\delta$ denote functions defined by the recipe
\begin{equation}
\iota(\Gamma(\alpha,\beta))=\langle m(\alpha,\beta),\delta(\alpha,\beta)\rangle.
\end{equation}

Following~\cite{nsbpr} and~\cite{819}, given an natural number $m$, we let $c_m(\alpha,\beta)$ be defined as $\beta^{m}_i(\alpha,\beta)$, where $i$ is the least number for which
\begin{equation}
\Gamma(\alpha,\beta)\neq\Gamma(\alpha,\beta^{m}_i(\alpha,\beta)).
\end{equation}
Note that $c_m(\alpha,\beta)$ is always defined when $\alpha<\beta$ because of our convention that $\Gamma(\alpha,\alpha)$ is equal to $\infty$.

We now define
\begin{equation}
\beta^*(\alpha,\beta)=c_{m(\alpha,\beta)}(\alpha,\beta)
\end{equation}
and
\begin{equation}
\eta^*(\alpha,\beta)=\eta^{m(\alpha,\beta)}(\beta^*(\alpha,\beta),\beta),
\end{equation}
where $\eta^m$ is as in~(\ref{etadef}), interpreted in the context of a generalized $C$-system.

If $\eta^*(\alpha,\beta)<\alpha$, we define
\begin{equation}
\alpha^*(\alpha,\beta)= c_{m(\alpha,\beta)}(\eta^*(\alpha,\beta),\alpha),
\end{equation}
and then set
\begin{equation}
D(\alpha,\beta):=\langle\alpha^*(\alpha,\beta),\beta^*(\alpha,\beta),\delta(\alpha,\beta)\rangle.
\end{equation}
If $\eta^*(\alpha,\beta)\geq\alpha$ we define $D(\alpha,\beta)$ arbitrarily.

The informal description associated with the preceding definition is much clearer than the notation
 required to write it down precisely. Given $\alpha<\beta$, we compute $D(\alpha,\beta)$ in the following manner.
   First, we use $\Gamma$ and $\iota$ to obtain a natural number $m$ and an ordinal $\delta<\cf(\mu)$.
      The number~$m$ tells us which piece of the generalized $C$-system $\bar{e}$ we will be using for our minimal walks,
       while $\delta$ appears in the final output of $D$.  The next step is to take the $m$-walk from $\beta$ down to
        $\alpha$ until we reach a spot $\beta^*$ where $\Gamma(\alpha,\beta^*)$ is different from
         $\Gamma(\alpha,\beta)$ --- this is the ordinal $\beta^*(\alpha,\beta)$.
           Given $\beta^*$, we proceed as in the preceding section and isolate the ordinal $\eta^m(\beta^*,\beta)$,
            which we name $\eta^*=\eta^*(\alpha,\beta)\leq\beta^*$.  If it happens that $\eta^*<\alpha$, then we
             $m$-walk from $\alpha$ down to $\eta^*$ until we reach a point $\alpha^*$ where $\Gamma(\eta^*,\alpha^*)$
              is different from $\Gamma(\eta^*,\alpha)$. The function $D$ now returns the value
               $\langle\alpha^*,\beta^*,\delta\rangle$.

The rest of the proof consists in showing that the function $D$ has the required properties,
 so let us assume $\langle t_\alpha:\alpha<\mu^+\rangle$ is a pairwise disjoint collection of members
  of $[\mu^+]^{<\cf(\mu)}$.  After a bit of culling and re-indexing, we may assume that
\begin{gather}
\alpha\leq\min(t_\alpha),\\
\intertext{and}
\alpha<\beta\Longrightarrow\sup(t_\alpha)<\min(t_\beta).
\end{gather}

We now apply Lemma~\ref{mainlemma} to obtain $m<\omega$ and $k<\cf(\mu)$ such that
\begin{multline}
(\exists^*\eta<\mu^+)(\exists\eta^+<\mu^+)(\exists^{\stat}\beta^*<\mu^+)\\
(\exists^*\gamma<\mu^+)(\exists\beta<\mu^+)[\psi_{k,m}(\eta,\eta^+,\beta^*,\gamma,\beta)],
\end{multline}
and then fix ordinals $\eta_a\leq\eta_a^+$ such that
\begin{equation}
(\exists^{\stat}\alpha^*<\mu^+)\\
(\exists^*\gamma<\mu^+)(\exists\alpha<\mu^+)[\psi_{k,m}(\eta_a,\eta_a^+,\alpha^*,\gamma,\alpha)].
\end{equation}

The next definition and claim are quite technical, but they are critical for our argument.

\begin{definition}
We say that $\{\eta_b, \eta_b^+\}$ is an {\em $\eta$-candidate} if
\begin{itemize}
\item $\eta_a^+<\eta_b\leq\eta^+_b<\mu^+$, and
\sk
\item $(\exists^{\stat}\beta^*<\mu^+)(\exists^*\gamma<\mu^+)(\exists\beta<\mu^+)[\psi_{k, m}(\eta_b,\eta_b^+,\beta^*,\gamma,\beta)]$.
\sk
\end{itemize}
We say that $\{\eta_b,\eta_b^+\}$ {\em works for} $\{\alpha^*,\alpha\}$ at $i_a$ if
\begin{itemize}
\sk
\item $\eta_b\leq\eta_b^+<\alpha^*<\min(t_\alpha)$,
\sk
\item $i_a<\cf(\mu)$,
\sk
\item $\Gamma^+(\eta,\beta^m_i(\alpha^*,\epsilon))=i_a$ for all $\epsilon\in t_\alpha$, $i<\rho^m_2(\alpha^*,\epsilon)$,  and $\eta\in\{\eta_b,\eta^+_b\}$, and
    \sk
\item $f_{\eta_b}(i_a+1)>f_{\alpha^*}(i_a+1)$.
\end{itemize}
\end{definition}

The next lemma says that there are many $\eta$-candidates that will be suitable for our construction.

\begin{lemma}
\label{uglyclaim}
For all sufficiently large $i_a<\cf(\mu)$, we can find an $\eta$-candidate $\{\eta_b,\eta_b^+\}$ such that
\begin{multline}
(\exists^{\stat}\alpha^*<\mu^+)(\exists^*\gamma<\mu^+)(\exists\alpha<\mu^+)\\
[\psi_{k, m}(\eta_a,\eta_a^+,\alpha^*,\gamma,\alpha)\text{ and $\{\eta_b,\eta^+_b\}$ works for $\{\alpha^*,\alpha\}$ at $i_a$}].
\end{multline}
\end{lemma}
\begin{proof}
Our choice of $k$ and $m$ tells us that there are unboundedly many $\eta$ for which we can find an $\eta^+$ such that $\{\eta,\eta^+\}$ is an $\eta$-candidate, so we can construct a sequence $\langle s_\alpha:\alpha<\mu^+\rangle$ of pairwise disjoint subsets of $\mu^+$ such that each $s_\alpha$ is an $\eta$-candidate.

Let $M_0\in M_1\in M_2$ be elementary submodels of $\mathfrak{A}$ as in Lemma~\ref{Gammalemma}, chosen so that
\begin{equation}
  \{\eta_a,\eta_a^+,\langle s_\alpha:\alpha<\mu^+\rangle\}\in M_0,
\end{equation}
 and, for $\alpha^*=M_0\cap\mu^+$,
\begin{equation}
\label{statmany}
(\exists^*\gamma<\mu^+)(\exists\alpha<\mu^+)[\psi_{k,m}(\eta_a,\eta_a^+,\alpha^*,\gamma,\alpha)].
\end{equation}
This can be done because there are stationarily~$\alpha^*$ satisfying~(\ref{statmany}).

Choose $\gamma\geq M_2\cap\mu^+$ and $\alpha<\mu^+$ such that $\psi_{k,m}(\eta_a,\eta_a^+,\alpha^*,\gamma,\alpha)$ holds, and define
\begin{equation}
t:=\{\beta^m_i(\alpha^*,\epsilon_a):\epsilon_a\in t_\alpha\text{ and }i<\rho^m_2(\alpha^*,\epsilon_a)\}.
\end{equation}
We can now apply Lemma~\ref{Gammalemma} to conclude that for all sufficiently large $i_a<\cf(\mu)$, there is an $\alpha<\alpha^*$ such that
\begin{equation}
\Gamma^+(\eta,\epsilon)=i_a\text{ for all $\eta\in s_\alpha$ and $\epsilon\in t$},
\end{equation}
and
\begin{equation}
f_{\alpha^*}(i_a+1)<f_{\eta}(i_a+1)\text{ for all $\eta\in s_\alpha$}.
\end{equation}
It should be clear that $\{\eta_b,\eta^+_b\}$ works for $\{\alpha^*,\alpha\}$ at $i_a$.  Since $\{\eta_b,\eta^+_b\}\in M_0$, the conclusion of the lemma now follows by a standard elementary submodel argument like that used to finish the proof of Lemma~\ref{mainlemma}.
\end{proof}

In light of the preceding claim,  we can fix $i_a>k$ and an $\eta$-candidate $\{\eta_b,\eta^+_b\}$ such that
\begin{multline}
(\exists^{\stat}\alpha^*<\mu^+)(\exists^*\gamma<\mu^+)(\exists\alpha<\mu^+)\\
[\psi_{k, m}(\eta_a,\eta_a^+,\alpha^*,\gamma,\alpha)\text{ and $\{\eta_b,\eta_b^+\}$ works for $\{\alpha^*,\alpha\}$ at $i_a$}].
\end{multline}

Let us now define
\begin{multline*}
S^:=\{\alpha^*<\mu^+:(\exists^*\gamma<\mu^+)(\exists\alpha<\mu^+)\\
[\psi_{k, m}(\eta_a,\eta_a^+,\alpha^*,\gamma,\alpha)\text{ and $\{\eta_b,\eta_b^+\}$ works for $\{\alpha^*,\alpha\}$ at $i_a$}]\}
\end{multline*}
and
\begin{equation*}
T^*:=\{\beta^*<\mu^+:(\exists^*\gamma<\mu^+)(\exists\beta<\mu^+)[\psi_{k, m}(\eta_b,\eta_b^+,\beta^*,\gamma,\beta)].
\end{equation*}

Our choices make it clear that $S^*$ and $T^*$ are both stationary. We will thin out these sets a bit to obtain the promised stationary sets $S$ and $T$.  To do this, let $x$ consist of those parameters needed to comprehend $\psi_{k,m}$ (see the first line of the proof of Lemma~\ref{mainlemma}), and let
\begin{equation*}
y:= x\cup\{S^*,T^*,\eta_a,\eta^+_a,\eta_b,\eta_b^+\}.
\end{equation*}
Let $\langle M_\delta:\delta<\mu^+\rangle$ be a $\mu^+$-approximating sequence over $y$, and let
\begin{equation*}
E:=\{\delta<\mu^+:M_\delta\cap\mu^+=\delta\}.
\end{equation*}
We define
\begin{equation*}
S:=S^*\cap E,
\end{equation*}
and
\begin{equation*}
T:=T^*\cap E.
\end{equation*}

Now suppose $\langle\alpha^*,\beta^*,\delta\rangle\in S\circledast T\times\cf(\mu)$.
We must produce $\alpha<\beta$ such that $D\restr t_\alpha\times t_\beta$ is constant with value $\langle\alpha^*,\beta^*,\delta\rangle$. To this point, we know
\begin{equation}
\eta_a\leq\eta_a^+<\eta_b\leq\eta_b^+<\alpha^*<\beta^*.
\end{equation}
We choose now $\gamma_b$ and $\beta$ such that
\begin{itemize}
\sk
\item $\psi_{k, m}(\eta_b,\eta_b^+,\beta^*,\gamma_b,\beta)$, and
\sk
\item $M_{\beta^*+2}\cap\mu^+\leq\gamma_b$.
\sk
\end{itemize}
This can be done by the definition of $T$.  Next, we define
\begin{equation}
t:=\{\beta^m_i(\beta^*,\epsilon_b):\epsilon_b\in t_\beta\text{ and }i<\rho^m_2(\beta^*,\epsilon_b)\}.
\end{equation}
By definition, $\gamma_b=\min(t)$ and so $M_{\beta^*+2}\cap\mu^+\leq\min(t)$.

Finally, define $J$ to be the set of all $\alpha<\mu^+$ such that for some $\gamma<\mu^+$, we have
\sk
\begin{itemize}
\item $\psi_{k, m}(\eta_a,\eta_a^+,\alpha^*,\gamma,\alpha)$, and
\sk
\item $\{\eta_b,\eta_b^+\}$ works for $\{\alpha^*,\alpha\}$ at $i_a$.
\sk
\end{itemize}

It is clear that $J$ is unbounded in~$\mu^+$ since $\alpha^*\in S$. Furthermore, the set $J$ is definable in the model $M_{\alpha^*+1}$, hence $J\in M_{\beta^*}$.

 Thus, the objects $M_{\beta^*}\in M_{\beta^*+1}\in M_{\beta^*+2}$, $\beta^*$, $\langle t_\alpha:\alpha\in J\rangle$, and $t$ satisfy the hypotheses of Lemma~\ref{Gammalemma}.

We conclude that for all sufficiently large $i_b<\cf(\mu)$, there are unboundedly many  $\alpha\in J\cap\beta^*$ such that for all $\epsilon_a\in t_\alpha$ and all $\epsilon\in t$,
\begin{equation}
\label{5.23}\Gamma^+(\epsilon_a,\epsilon)=i_b,
\end{equation}
and
\begin{equation}
\label{5.24}f_{\beta^*}(i_b+1)<f_{\epsilon_a}(i_b+1).
\end{equation}
In particular, we can choose $\alpha$ and $i_b$ in such a way that the above conditions are satisfied, and in addition such that
\begin{equation}
\iota(i_b)=\langle m, \delta\rangle
\end{equation}
Notice that for any $\epsilon_a\in t_\alpha$ and $\epsilon_b\in t_\beta$, we have
\begin{equation}
\eta_a\leq\eta^+_a<\eta_b\leq\eta_b^+<\alpha^*<\epsilon_a<\beta^*<\epsilon_b.
\end{equation}
The rest of the proof consists of show that $D\restr t_\alpha\times t_\beta$ is constant with value $\langle\alpha^*,\beta^*,\delta\rangle$, so assume now that $\epsilon_a\in t_\alpha$ and $\epsilon_b\in t_\beta$.

Right away, we see that $\Gamma(\epsilon_a,\epsilon_b)=i_b$ since $\epsilon_a\in t_\alpha$ and $\epsilon_b\in t$.
As an immediate corollary, it follows that
\begin{gather}
m(\epsilon_a,\epsilon_b)=m\\
\intertext{and}
\label{deltastar}\delta(\epsilon_a,\epsilon_b)=\delta.
\end{gather}

\begin{claim}
\label{betastar}
$\beta^*(\epsilon_a,\epsilon_b)=\beta^*.$
\end{claim}
\begin{proof}
Since $\psi_{k, m}(\epsilon_b,\epsilon_b^+,\beta^*,\gamma_b,\beta)$ holds, we know
\begin{equation}
\eta_b\leq\eta^m(\beta^*,\epsilon_b)\leq\eta_b^+.
\end{equation}
Since
\begin{equation}
\eta_b^+<\alpha^*<\epsilon_a<\beta^*,
\end{equation}
it follows that
\begin{equation}
\beta^m_i(\epsilon_a,\epsilon_b)=\beta^m_i(\beta^*,\epsilon_b)\text{ for }i\leq\rho^m_2(\beta^*,\epsilon_b).
\end{equation}
In particular,
\begin{equation}
\beta^m_i(\epsilon_a,\epsilon_b)\in t\text{ for }i<\rho^m_2(\beta^*,\epsilon_b),
\end{equation}
and
\begin{equation}
\beta^m_{\rho^m_2(\beta^*,\epsilon_b)}(\epsilon_a,\epsilon_b)=\beta^*.
\end{equation}
Given our choice of $\alpha$ (see~(\ref{5.23}) and~(\ref{5.24})), we obtain
\begin{equation}
\beta^*(\epsilon_a,\epsilon_b)=c_m(\epsilon_a,\epsilon_b)=\beta^*,
\end{equation}
as required.
\end{proof}

Note that the preceding claim tells us
\begin{equation}
\eta^*(\epsilon_a,\epsilon_b)=\eta^m(\beta^*,\epsilon_b)
\end{equation}
as well.

\begin{claim}
\label{alphastar}
$\alpha^*(\epsilon_a,\epsilon_b)=\alpha^*$.
\end{claim}
\begin{proof}
Our choice of $\{\eta_b,\eta_b^+\}$ tells us that
\begin{equation*}
\eta_a\leq\eta_a^+<\eta_b\leq\eta^*(\epsilon_a,\epsilon_b)=\eta^m(\beta^*,\epsilon_b)\leq\eta_b^+<\alpha^*<\epsilon_a,
\end{equation*}
and therefore
\begin{equation}
\label{agree}
\beta^m_i(\eta^*(\epsilon_a,\epsilon_b),\epsilon_a)=\beta^m_i(\alpha^*,\epsilon_a)\text{ for all }i\leq\rho^m_2(\alpha^*,\epsilon_a).
\end{equation}
Furthermore, we know
\begin{equation}
\label{restrictions}
f_{\eta_b}\restr [k,\cf(\mu))\leq f_{\eta^*(\epsilon_a,\epsilon_b)}\restr [k,\cf(\mu))\leq f_{\eta^+}\restr [k,\cf(\mu)),
\end{equation}
and this implies
\begin{gather}
\label{part1} f_{\eta_b}(i_a)\leq f_{\eta^*(\epsilon_a,\epsilon_b)}(i_a)\\
\label{part2} f_{\eta_b}(i_a+1)\leq f_{\eta^*(\epsilon_a,\epsilon_b)}(i_a+1)\\
\intertext{and}
\label{part3} f_{\eta^*(\epsilon_a,\epsilon_b)}\restr [i_a+1,\cf(\mu))\leq f_{\eta_b^+}\restr [i_a+1,\cf(\mu)).
\end{gather}

For $i<\rho^m_2(\alpha^*,\epsilon_a)$, we know that $\{\eta_b,\eta_b^+\}$ works for $\{\alpha^*,\alpha\}$ at~$i_a$, and
so
\begin{equation}
\label{thirdwing}
\Gamma^+\bigl(\eta_b,\beta^m_i(\alpha^*,\epsilon_a)\bigr)=\Gamma^+\bigl(\eta_b^+,\beta^m_i(\alpha^*,\epsilon_a)\bigr)=i_a,
\end{equation}
and
\begin{equation}
\label{fourthwing}
f_{\alpha^*}(i_a+1)<f_{\eta^b}(i_a+1).
\end{equation}

From~(\ref{agree}), (\ref{restrictions}), (\ref{part1}), (\ref{part3}), and~(\ref{thirdwing}) we conclude
\begin{equation}
\label{firstwing}
\Gamma\bigl(\eta^*(\epsilon_a,\epsilon_b),\beta^m_i(\eta^*(\epsilon_a,\epsilon_b),\epsilon_a)\bigr)=i_a\text{ for all }i<\rho^m_2(\alpha^*,\epsilon_a).
\end{equation}

Since
\begin{equation}
\beta^m_{\rho^2_m(\alpha^*,\epsilon_a)}(\eta^*(\epsilon_a,\epsilon_b),\epsilon_a)=\alpha^*,
\end{equation}
it follows from~(\ref{part2}), (\ref{fourthwing}), and~(\ref{restrictions}) that
\begin{equation}
\label{secondwing}
\Gamma\bigl(\eta^*(\epsilon_a,\epsilon_b), \beta^m_{\rho^2_m(\alpha^*,\epsilon_a)}(\eta^*(\epsilon_a,\epsilon_b),\epsilon_a)\bigr)\neq i_a.
\end{equation}
From~(\ref{firstwing}), (\ref{secondwing}), and the definition of $c_{m(\epsilon_a,\epsilon_b)}$, we conclude that
\begin{equation}
\alpha^*(\epsilon_a,\epsilon_b)=c_{m(\epsilon_a,\epsilon_b)}(\eta^*(\epsilon_a,\epsilon_b),\epsilon_a)
=\beta^m_{\rho^2_m(\alpha^*,\epsilon_a)}(\eta^*(\epsilon_a,\epsilon_b),\epsilon_a)=\alpha^*,
\end{equation}
as required.
\end{proof}

Putting Claim~\ref{betastar}, Claim~\ref{alphastar}, and~(\ref{deltastar}) together, we see
\begin{equation}
D(\alpha,\beta)=\langle\alpha^*,\beta^*,\delta\rangle
\end{equation}
and the proof is complete.
\end{proof}
\section{Consequences}

We turn our attention now to consequences of Theorem~\ref{mainthm} and pick up the discussion of the introduction once more.  Our first result gives the promised equivalence of~(\ref{first1}) and~(\ref{second1}), and also establishes an even stronger fact.

\begin{theorem}
\label{upgradetheorem}
The following are equivalent for a singular cardinal~$\mu$:
\begin{align}
\label{mu+}&\pr_1(\mu^+,\mu^+,\mu^+,\cf(\mu))\hphantom{xxxxxxxxxxxxxxxxxxxxxxxxxxxxxxxxxxxxxxxx}\\
\label{mu}&\pr_1(\mu^+,\mu^+,\mu,\cf(\mu))\\
\label{theta}&\pr_1(\mu^+,\mu^+,\theta,\cf(\mu))\text{ for arbitrarily large }\theta<\mu.
\end{align}

\end{theorem}
\begin{proof}
It is easy to see that each statement implies the one following, so we prove that~(\ref{theta}) implies~(\ref{mu+}).
 Fix an increasing sequence of cardinals $\langle \theta_i:i<\cf(\mu)\rangle$ cofinal in $\mu$ such that
  $\pr_1(\mu^+,\mu^+,\theta_i,\cf(\mu))$ holds for each $i<\cf(\mu)$, and let $c_i$ be a coloring witnessing this for $\theta_i$.  Also, we fix for each $\beta<\mu^+$ a function $g_\beta$ mapping $\mu$ onto $\beta$.

Let $D$ be a coloring as in our main theorem, and let $\alpha^*$, $\beta^*$, and $\delta$ denote functions defined by the recipe
\begin{equation}
D(\alpha,\beta)=\langle \alpha^*(\alpha,\beta),\beta^*(\alpha,\beta),\delta(\alpha,\beta)\rangle.
\end{equation}
Finally, define the function $c:[\mu^+]^2\rightarrow \mu^+$ as follows:
\begin{equation}
\label{cdef}
c(\alpha,\beta):= g_{\beta^*(\alpha,\beta)}\left(c_{\delta(\alpha,\beta)}\left(\alpha^*(\alpha,\beta),\beta^*(\alpha,\beta)\right)\right).
\end{equation}

Now suppose $\langle t_\beta:\beta<\mu^+\rangle$ is a sequence of pairwise disjoint elements of $[\mu^+]^{<\cf(\mu)}$, and let $\varsigma<\mu^+$ be arbitrary. We must find $\alpha<\beta$ such that $c\restr t_\alpha\times t_\beta$ is constant with value $\varsigma$.

Fix stationary sets $S$ and $T$ as in the conclusion of our main theorem. An application of Fodor's Theorem allows us to find $\epsilon<\mu$ and stationary $T^*\subs T$ such that $g_{\beta^*}(\epsilon)=\varsigma$ for all $\beta^*\in T^*$.

Next, we construct a sequence $\langle s_\gamma:\gamma<\mu^+\rangle$ of pairwise disjoint elements of $S\circledast T$ such that $\max(s_\zeta)<\min(s_\eta)$ whenever $\zeta<\eta$. This is easily done, as both $S$ and $T^*$ are unbounded in $\mu^+$. Now let $\delta<\cf(\mu)$ be chosen so that $\epsilon<\theta_\delta$.

Our choice of $c_\delta$ provides us with $\zeta<\eta<\mu^+$ such that
\begin{equation}
c_\delta\restr s_\zeta\times s_\eta\text{ is constant with value }\epsilon.
\end{equation}
Now supposing
\begin{gather}
s_\zeta = \{\alpha^*_\zeta,\beta^*_\zeta\}
\intertext{and}
s_\eta = \{\alpha^*_\eta,\beta^*_\eta\},
\end{gather}
we define
\begin{gather}
\alpha^*=\alpha^*_\zeta
\intertext{and}
\beta^*=\beta^*_\eta
\end{gather}
It should be clear that $\langle\alpha^*,\beta^*\rangle\in S\circledast T$.

Our assumptions about $D$ now give us $\alpha<\beta$ such that
\begin{equation}
D\restr t_\alpha\times t_\beta\text{ is constant with value }\langle\alpha^*,\beta^*,\delta\rangle.
\end{equation}
Clearly we can also demand that $\sup(t_\alpha)<\min(t_\beta)$, and now we show
\begin{equation}
c\restr t_\alpha\times t_\beta\text{ is constant with value }\varsigma.
\end{equation}

Given $\epsilon_a\in t_\alpha$ and $\epsilon_b\in t_\beta$, we know
\begin{gather}
\alpha^*(\epsilon_a,\epsilon_b)=\alpha^*,\\
\beta^*(\epsilon_a,\epsilon_b)=\beta^*,\\
\intertext{and}
\delta(\epsilon_a,\epsilon_b)=\delta.
\end{gather}
A glance at~(\ref{cdef}) tells us
\begin{equation}
c(\epsilon_a,\epsilon_\beta)= g_{\beta^*}(c_\delta(\alpha^*,\beta^*)),
\end{equation}
and now the result follows immediately as $\beta^*\in T^*$ and $c_\delta(\alpha^*,\beta^*)=\epsilon$.
\end{proof}

The main theorem of our paper~\cite{ideals} established, among other things, that if $\mu$ is singular of uncountable
 cofinality, then $\pr_1(\mu^+,\mu^+,\mu^+,\cf(\mu))$ holds unless the stationary subsets of $\mu^+$ possess many instances of
 stationary reflection.  In the case where the cofinality of~$\mu$ is countable, we were only able to get the analogous result
 with $\pr_1(\mu^+,\mu^+,\mu,\cf(\mu))$, but this defect is repaired now by the equivalence of~(\ref{mu}) and~(\ref{mu+}).
This allows to state the following theorem without restrictions on the cofinality of~$\mu$:

\begin{theorem}
If $\mu$ is singular and $\pr_1(\mu^+,\mu^+,\mu^+,\cf(\mu))$ fails then there is a $\theta<\mu$ such that
for any sequence $\langle S_\alpha:\alpha<\sigma\rangle$ of stationary subsets of $S^{\mu^+}_{\geq\theta}$ of length $\sigma<\cf(\mu)$,
there is an ordinal $\delta<\mu^+$ such that $S_\alpha\cap\delta$ is stationary in $\delta$ for all $\alpha<\sigma$.
\end{theorem}
\begin{proof}
This is restatement the contrapositive of parts (2) and (3) of the main theorem of~\cite{ideals}, in light of the equivalence of~(\ref{mu}) and~(\ref{mu+}).
\end{proof}

It is still open whether $\pr_1(\mu^+,\mu^+,\mu^+,\cf(\mu))$ can fail for a singular cardinal. The above theorem tells us that obtaining such a consistency result will necessarily involve some considerable large cardinals.  In light of the implications in~(\ref{implications}), we see that this is true for the consistency of $\mu^+$ being Jonsson, or the consistency of $\mu^+\rightarrow[\mu^+]^2_{\mu^+}$ as well.

Our next result is a relative of one of the conclusions derived in~\cite{upgradesi}.

\begin{theorem}
\label{stattheorem}
The following are equivalent for a singular cardinal~$\mu$ and cardinal $\theta\leq\mu^+$:
\begin{enumerate}
\item $\pr_1(\mu^+,\mu^+,\theta,\cf(\mu))$.
\sk
\item There is a function $c:[\mu^+]^2\rightarrow\theta$ such that for any unbounded subsets $A$ and $B$ of $\mu^+$, \begin{equation}
    \label{abuse1}
    \theta\subs\ran(c\restr A\circledast B).
    \end{equation}
\item There is a function $d:[\mu^+]^2\rightarrow\theta$ such that for any stationary subsets $S$ and $T$ of $\mu^+$,
    \begin{equation}
    \label{abuse}
    \theta\subs\ran(d\restr S\circledast T).
    \end{equation}
\end{enumerate}
\end{theorem}
We are abusing notation a little bit in~(\ref{abuse1}) and~(\ref{abuse}), as elements of $S\circledast T$ are technically ordered pairs and not pairs of ordinals, but the meaning should be clear.  Also note that~(2) is relative of the relation
\begin{equation}
\label{prel}
\mu^+\nrightarrow[(\mu^+:\mu^+)]^2_\theta
 \end{equation}
 (see~\cite{ehmr}), which states that there is a function $f:[\mu^+]^2\rightarrow\theta$ such that for any unbounded subsets $A$ and $B$ of $\mu^+$ and any $\varsigma<\theta$, there are $\alpha\in A$ and $\beta\in B$ with $f(\alpha,\beta)=\varsigma$. The difference between~(2) and~(\ref{prel}) is very slight --- in~(\ref{prel}), it is not
 required that $\alpha$ is less than $\beta$, while we need this in order to apply Theorem~\ref{mainthm}.

\begin{proof}
The fact that (1) implies (2) is well-known (we did something similar in the proof of Theorem~\ref{upgradetheorem}), but we give it for completeness.  We show something a little bit stronger, namely that any function witnessing~(1) also works for~(2).

Thus, let $c$ witness that $\pr_1(\mu^+,\mu^+,\theta,\cf(\mu))$ holds, let $\varsigma<\theta$ be given,  and let $S$ and $T$ be stationary subsets of~$\mu^+$.  Construct a sequence $\langle t_\alpha:\alpha<\mu^+\rangle$ of pairwise disjoint elements of $S\circledast T$  such that $\max(t_\alpha)<\min(t_\beta)$ whenever $\alpha<\beta$, and then fix $\alpha<\beta$ such that
$c\restr t_\alpha\times t_\beta$ is constant with value~$\varsigma$. Let $\alpha^*=\min(t_\alpha)$ and $\beta^*=\max(t_\beta)$. Then $\langle\alpha^*,\beta^*\rangle\in S\circledast T$ is as required.

It is clear that (2) implies (3), so assume $d$ be as in~(3), and let $D$ be the function from Theorem~\ref{mainthm}, with
\begin{equation}
D(\alpha,\beta)=\langle\alpha^*(\alpha,\beta),\beta^*(\alpha,\beta),\delta(\alpha,\beta)\rangle.
\end{equation}
We define a function $f:[\mu^+]^2\rightarrow\theta$ by
\begin{equation}
f(\alpha,\beta)=c(\alpha^*(\alpha,\beta),\beta^*(\alpha,\beta)),
\end{equation}
and the verification that $f$ has the required properties is straightforward.
\end{proof}

We will present only one application of the preceding theorem here, but we note that we can use Theorem~\ref{stattheorem} to greatly simplify the proof of the main result of~\cite{535}. It also allows us
to solve the main problem left open by~\cite{819}. We intend to present this work elsewhere, as it is joint with Shelah.

We start with a lemma, proved by a standard argument

\begin{lemma}
\label{otherscale}
Let $(\vec{\mu},\vec{f})$ be a scale for the singular cardinal~$\mu$, and suppose $A\subs\mu^+$ is unbounded.
Then
\begin{equation}
(\forall^*i<\cf(\mu))(\exists^*\xi<\mu_i)(\exists^*\alpha\in A)[f_\alpha(i)=\xi].
\end{equation}
\end{lemma}
\begin{proof}
The proof is by contradiction, so suppose the conclusion fails for some unbounded $A\subs\mu^+$. Parsing what this means,
we find that there are unboundedly many $i<\cf(\mu)$, for all sufficiently large $\xi<\mu_i$, the set of $\alpha\in A$ with $f_\alpha(i)=\xi$
is bounded in $\mu^+$.

Let $I$ consist of those $i<\cf(\mu)$ for which the above is true, and for $i\in I$ choose $\xi_i$ such that
\begin{equation}
\xi_i\leq\xi<\mu_i\Longrightarrow |\{\alpha\in A:f_\alpha(i)=\xi\}|<\mu^+.
\end{equation}

Given $i\in I$ and $\xi$ with $\xi_i\leq\xi<\mu_i$, we can fix an ordinal $\alpha(\xi,i)<\mu^+$ such that
\begin{equation}
\{\alpha\in A: f_\alpha(i)=\xi\}\subs \alpha(\xi, i),
\end{equation}
and then define
\begin{equation}
\alpha^*:=\sup\{\alpha(\xi, i):i<\cf(\mu), \xi_i\leq\xi<\mu_i\}.
\end{equation}
It is clear that $\alpha^*<\mu^+$ as there are only $\mu$ possibilities for $\xi$ and $i$.

After the dust has settled, we see that if $\alpha\in A$ is greater than $\alpha^*$, then
\begin{equation}
i\in I\Longrightarrow f_\alpha(i)<\xi_i,
\end{equation}
and this easily contradicts our assumption that $(\vec{\mu},\vec{f})$ is a scale.
\end{proof}

The theorem we prove below has many antecedents in the literature, but our result seems to be the first in which the partition relation holding at the successor of the singular cardinal represents an upgrade over those assumed to hold at the smaller cardinals.

\begin{theorem}
\label{lastthm}
Suppose $\mu$ is singular, and there is a scale $(\vec{\mu},\vec{f})$ for $\mu$ such that
\begin{equation}
\label{colonrelation}
\mu_i\nrightarrow[(\mu_i:\mu_i)]^2_{\mu_i}
\end{equation}
for all $i<\cf(\mu)$.
Then $\pr_1(\mu^+,\mu^+,\mu^+,\cf(\mu))$ holds.
\end{theorem}

\begin{proof}
By Conclusion 4.1A on page~67 of~\cite{cardarith}, we know that $\pr_1(\mu^+,\mu^+,\cf(\mu),\cf(\mu))$ holds, so we can fix a function $c:[\mu^+]^2\rightarrow\cf(\mu)$ as in part~(2) of Theorem~\ref{stattheorem} with $\cf(\mu)$ standing in for~$\theta$.

For each $i<\cf(\mu)$, let $d_i:[\mu_i]^2\rightarrow\mu_i$ be a witness for~(\ref{colonrelation}), and define a function
$d:[\mu^+]^2\rightarrow\mu$ by
\begin{equation}
d(\alpha,\beta)=d_{c(\alpha,\beta)}(f_\alpha(c(\alpha,\beta)),f_\beta(c(\alpha,\beta))).
\end{equation}

Given $\varsigma<\mu$ and unbounded subsets $A$ and $B$ of $\mu^+$, we will find $\langle\alpha,\beta\rangle$ in $A\circledast B$
 with $d(\alpha,\beta)=\varsigma$.  This implies $\pr_1(\mu^+,\mu^+,\mu,\cf(\mu))$ by Theorem~\ref{stattheorem}, which
 in turn gives us $\pr_1(\mu^+,\mu^+,\mu^+,\cf(\mu))$ by Theorem~\ref{upgradetheorem}.

We choose $i^*<\cf(\mu)$ so that
\begin{gather}
\varsigma<\mu_{i^*},\\
(\exists^*\zeta<\mu_{i^*})(\exists^*\alpha\in A)[f_\alpha(i^*)=\zeta],\\
\intertext{and}
(\exists^*\eta<\mu_{i^*})(\exists^*\beta\in B)[f_\alpha(i^*)=\eta].
\end{gather}
Clearly this is possible by Lemma~\ref{otherscale}.

Next, we define
\begin{gather}
A^*:=\{\zeta<\mu_{i^*}:(\exists^*\alpha\in A)[f_\alpha(i^*)=\zeta]\}\\
\intertext{and}
B^*:=\{\eta<\mu_{i^*}:(\exists^*\beta\in B)[f_\alpha(i^*)=\eta]\}.
\end{gather}

Both of these sets are unbounded in~$\mu_{i^*}$, and so we can find $\alpha^*\in A^*$ and $\beta^*\in B^*$ with
\begin{equation}
d_{i^*}(\alpha^*,\beta^*)=\varsigma.
\end{equation}

Now define
\begin{gather}
A^\dagger=\{\alpha\in A: f_\alpha(i^*)=\alpha^*\}\\
\intertext{and}
B^\dagger=\{\beta\in B: f_\beta(i^*)=\beta^*\}.
\end{gather}

Both of these sets are unbounded in~$\mu^+$, and so we can find $\langle \alpha,\beta\rangle$ in $A\circledast B$ with $c(\alpha,\beta)=i^*$.

We find now that
\begin{multline}
d(\alpha,\beta)=d_{c(\alpha,\beta)}(f_\alpha(c(\alpha,\beta)),f_\beta(c(\alpha,\beta)))\\
=d_{i^*}(f_\alpha(i^*),f_\beta(i^*))=d_{i^*}(\alpha^*,\beta^*)=\varsigma,
\end{multline}
as required.
\end{proof}

We come now to a result promised at the end of the introduction. The hypothesis $\pp(\mu)=\mu^+$ refers to Shelah's pseudo-power function, but we will not elaborate as we need only one easily understood consequence of this assumption. The reader can consult~\cite{cardarith} for the definition of $\pp$, and the author's~\cite{myhandbook} contains the proof of the relevant facts.

\begin{corollary}
\label{lastcor}
$\pr_1(\mu^+,\mu^+,\mu^+,\cf(\mu))$ holds for any singular $\mu$ with $\pp(\mu)=\mu^+$.
\end{corollary}
\begin{proof}
Since $\pp(\mu)=\mu^+$, there is a scale $(\vec{\mu},\vec{f})$ such that for each $i<\cf(\mu)$,  $\mu_i=\kappa_i^{++}$ for some uncountable regular cardinal $\kappa_i$. By~\cite{572}, it follows that
\begin{equation}
\mu_i\nrightarrow[(\mu_i:\mu_i)]^2_{\mu_i}
\end{equation}
for each $i<\cf(\mu)$, and so we get what we need by way of Theorem~\ref{lastthm}.
\end{proof}

The proof of the preceding is deceptively short, for~\cite{572} is quite a difficult paper.  Shelah shows there that
\begin{equation}
\label{prkapp}
\pr_1(\kappa^{++},\kappa^{++},\kappa^{++},\kappa)
\end{equation}
holds for regular $\kappa$, but the argument for the implication ``$(1)\Longrightarrow(2)$'' in Theorem~\ref{stattheorem}
tells us that
\begin{equation}
\kappa^{++}\nrightarrow[(\kappa^{++}:\kappa^{++})]^2_{\kappa^{++}}
\end{equation}
for every regular $\kappa$ as well.

Corollary~\ref{lastcor} can also be proved in the following manner.  Claim~4.1E on page~70 of~\cite{cardarith} implies, when suitably interpreted and combined with the result quoted in~(\ref{prkapp}), that $\pr_1(\mu^+,\mu^+,\mu,\cf(\mu))$ holds. This fact was noted by Shelah in a personal communication with the author, but Theorem~\ref{upgradetheorem} is still necessary to obtain a coloring with $\mu^+$ colors.  We chose our approach because Theorem~\ref{lastthm} is of independent interest and of more general applicability.

\bibliographystyle{plain}

\end{document}